\documentclass[reqno]{amsart}
\usepackage{graphicx}
\theoremstyle{plain}
\newtheorem{thm}{\bf Theorem}[section]

\newtheorem{lem}[thm]{\bf Lemma}
\theoremstyle{remark}
\newtheorem{defn}[thm]{\bf Def{}inition}

\numberwithin{equation}{section}
\usepackage{enumerate}
\newcommand{\CC}{\mathbb{C}}
\newcommand{\NN}{\mathbb{N}}
\newcommand{\RR}{\mathbb{R}}
\newcommand{\KK}{\mathbb{K}}

\begin{document}
\baselineskip8pt

\title{ Approximative Atomic Systems for operators in Banach spaces}

\author[S. Jahan]{Shah Jahan}
\address{Shah Jahan, Department of Mathematics, Ramjas College,
University of Delhi, Delhi-110007, India}
\email{chowdharyshahjahan@gmail.com}

\begin{abstract}\baselineskip10pt
$K$-frames and atomic systems for an operator $K$ in Hilbert spaces were introduced by Gavruta \cite{12} and further studied by Xio, Zhu and Gavruta \cite{21}. In this paper, we have introduced the notion of an approximative atomic system for an operator $K$ in Banach spaces and obtained interesting results. A complete characterization of family of approximative local atoms of subspace of Banach space has been obtained. Also, a necessary and sufficient condition for the existence of an approximative atomic system for an operator $K$ is given. Finally, explicit methods are given for the construction of an approximative atomic systems for an operator $K$ from a given Bessel sequence and approximative $\mathcal{X}_d$-Bessel sequence.
\end{abstract}



\maketitle \thispagestyle{empty} \baselineskip12pt
\section{Introduction}
A sequence $\{f_n\}_{n=1}^\infty$ in separable Hilbert space $\mathcal{H}$ is called a frame for the Hilbert space $\mathcal{H},$ if there exist positive constants $A,~ B >0$ such that
\begin{align}\label{0} A\| f\|^2_\mathcal{H} \leq \sum_{n=1}^\infty |\langle f, f_n \rangle|^2 \leq B\|f\|^2_\mathcal{H},~~~ \mbox{for all}~~ f\in \mathcal{H}
 \end{align} The positive constants $A$ and $B$ are called the lower and upper frame bounds of the frame, respectively. The inequality in (\ref{0}) is called the frame inequality of the frame. The frame is called a tight frame if $A=B$ and is called normalized tight frame if $A=B=1.$  If $\{f_n\}_{n=1}^\infty$  is a frame for $\mathcal{H}$ then the following operators are associated with it.
  \begin{enumerate}[(a)]
 \item Pre-frame operator $T: l^2(\mathbb{N}) \longrightarrow \mathcal{H}$ is defined as $T\{c_n\}_{n=1}^\infty= \sum\limits_{k=1}^\infty c_nf_n,~~\{c_n\}_{n=1}^\infty \in l^2(\mathbb{N}).$
 \item Analysis operator $T^*: \mathcal{H} \longrightarrow l^2(\mathbb{N}), T^*f= \{\langle f, f_k\rangle\}_{k=1}^\infty~~~ f \in \mathcal{H}.$
 \item Frame operator $S=TT^*=: \mathcal{H} \longrightarrow \mathcal{H},~~ Sf= \sum\limits_{k=1}^\infty \langle f, f_k \rangle f_k, ~~f \in \mathcal{H}.$ The frame operator $S$ is bounded, linear and invertible on $\mathcal{H}$. Thus, a frame for $\mathcal{H}$ allows each vector in $\mathcal{H}$ to be written as a linear combination of the elements in the frame, but the linear independence between the elements is not required; i.e for each vector $f \in \mathcal{H}$ we have,
     \begin{eqnarray*}
     f=SS^{-1}f= \sum\limits_{k=1}^\infty \langle f, f_k \rangle f_k.
     \end{eqnarray*}
   \end{enumerate}
Frames in Hilbert spaces were introduced by Duffin and Schaeffer \cite{8} in 1952, while addressing some deep problems in non-harmonic Fourier series. Frames were reintroduced by Daubechies, Grossmann and Meyer \cite{DGM} after three decades, in 1986, frames were brought to life, and were widely studied after this nobel work. Frames are generalization of orthonormal basis. The
main property of frames which makes them useful is their redundancy. Now, frames play an important role not only in  pure mathematics but also in applied mathematics. Representation
of signals using frames is advantageous over basis expansions in
a variety of practical applications in science and engineering. In
particular, frames are widely used in sampling theory \cite{ab, es}, wavelet theory \cite{id}, signal processing \cite{pg}, image processing \cite{ma}, pseudo-differential operators \cite{pdo},
filter banks \cite{cf}, quantum computing \cite{qc}, wireless sensor network \cite{ws}, coding theory \cite{ct}, geometry\cite{SV1, SV2} and so on.
 Feichtinger and Gr\"{o}cheing \cite{10} extended the notion of frames to Banach space and defined the notion of atomic decomposition.
 Gr\"{o}cheing \cite{13} introduced a more general concept for Banach spaces called Banach frame. Banach frames and atomic decompositions were further studied in \cite{6,13,14}. Casazza, Christensen and Stoeva \cite{4} studied $\mathcal{X}_d$-frame and $\mathcal{X}_d$-Bessel sequence in Banach spaces. Shah \cite{SS} defined and studied approximative $\mathcal{X}_d$-frames and approximative $\mathcal{X}_d$-Bessel sequences. He gave the following definition.
 \begin{defn}\cite{SS}
A sequence $ \{h_{n,i} \}\underset{n \in \mathbb{N}}{_{i=1,2,3,...,m_n}} \subseteq \mathcal{X}^*,$ where $\{m_n\}$ is an increasing sequence of positive integers, is called an approximative \emph{$\mathcal{X}_d$-frame} for $\mathcal{X}$ if
\begin{enumerate}[(a)]
\item $\{h_{n,i}(x) \}\underset{n \in \mathbb{N}}{_{i=1,2,3,...,m_n}} \in \mathcal{X}_d$, for all $ x\in \mathcal{X}.$
\item There exist constants $A$ and $B$ with $0 < A \leq B <\infty$ such that
\begin{eqnarray}\label{XD}
 A\|x \|_\mathcal{X}~\leq \| \{h_{n,i}(x) \}\underset{n \in \mathbb{N}}{_{i=1,2,3,...,m_n}} \|_{\mathcal{X}_d}~\leq B \|x \|_\mathcal{X},~\text{for all} \ x\in \mathcal{X}.
\end{eqnarray}
\end{enumerate}
\end{defn}
\noindent The constants $A$ and $B$ are called approximative \emph{$\mathcal{X}_d$-frame} bounds.
If atleast (a) and the upper bound condition in (\ref{XD}) are satisfied, then $\{h_{n,i}(x) \}\underset{n \in \mathbb{N}}{_{i=1,2,3,...,m_n}}$ is called an approximative \emph{$\mathcal{X}_d$-Bessel sequence} for $\mathcal{X}.$\\
One may note that if $\{f_n\}$ is an $\mathcal{X}_d$-frame for $\mathcal{X},$ then for $\{h_{n,i} \}=f_i,~i=1,2,3,...,n;~ n \in \mathbb{N}$, $\{h_{n,i} \}\underset{n \in \mathbb{N}}{_{i=1,2,3,...,m_n}}$ is an approximative $\mathcal{X}_d$-frame for $\mathcal{X}.$ Also, note that if $\{f_n\}$ is an $\mathcal{X}_d$-Bessel sequence for $\mathcal{X},$ then for $\{h_{n,i} \}=f_i,~i=1,2,3,...,n;~ n \in \mathbb{N}$, $\{h_{n,i} \}\underset{n \in \mathbb{N}}{_{i=1,2,3,...,m_n}}$ is an approximative $\mathcal{X}_d$-Bessel sequence for $\mathcal{X}.$ The
bounded linear operator $U: \mathcal{X} \rightarrow \mathcal{X}_d$ given by
\begin{eqnarray}
U(x)=\{h_{n,i}(x) \}\underset{n \in \mathbb{N}}{_{i=1,2,3,...,m_n}}, x \in \mathcal{X}
\end{eqnarray}
 is called an \emph{analysis operator} associated to the approximative $\mathcal{X}_d$-Bessel sequence $\{h_{n,i} \}\underset{n \in \mathbb{N}}{_{i=1,2,3,...,m_n}}$.
 If $\{h_{n,i} \}\underset{n \in \mathbb{N}}{_{i=1,2,3,...,m_n}}$ is an approximative \emph{$\mathcal{X}_d$-frame} for $\mathcal{X}$ and there exists a bounded linear operator $S:\mathcal{X}_d\longrightarrow \mathcal{X}$ such that $S(\{h_{n,i} \}\underset{n \in \mathbb{N}}{_{i=1,2,3,...,m_n}})=x,$~ for all $x\in \mathcal{X}$, then $(\{h_{n,i} \}\underset{n \in \mathbb{N}}{_{i=1,2,3,...,m_n}},S)$ is called an approximative \emph{Banach frame} for $\mathcal{X}$ with respect to $\mathcal{X}_d$.\\
 Atomic systems for an operator $K$ in Hilbert spaces were introduced by Gavruta \cite{12}.  Xiao et al. \cite{21} discussed relationship between $K$-frames and ordinary frames in Hilbert spaces. Poumai and Jahan \cite{KS} introduced atomic systems for operators in Banach spaces. Frames for operators in Banach spaces were further studied in \cite{R1, R3}.\\
 \textbf{Outline of the paper.}
$K$-frames and atomic systems for an operator $K$ in Hilbert spaces were introduced by Gavruta \cite{12} and further studied by Xiao, Zhu and Gavruta \cite{21} discussed relationship between $K$-frames and frames in Hilbert spaces.
 In the present paper, we define approximative atomic system for an operator $K$ in a Banach space  and prove some results on the existence of approximative atomic system for $K$. We also define approximative family of local atoms for subspaces and give a characterization for the approximative family of local atoms for subspaces. Also, we discuss  methods to construct approximative atomic system for an operator $K$ from an approximative Bessel sequence and an approximative $\mathcal{X}_d$-Bessel sequence.\\
\\\indent Throughout this paper, $\mathcal{X}$ will denote a Banach space over the scalar field $\KK$($\RR$ or $\CC$), $\mathcal{X}^*$ the dual space of $\mathcal{X}$, $\mathcal{X}_d$ a BK-space, $\mathcal{X}_d^*$ denotes dual of $\mathcal{X}_d$ and we will assume that $\mathcal{X}_d$ has a sequence of canonical unit vectors as basis. By $\{h_{n,i}\}$ we means a sequence of coefficient functionals of row finite matrix. $L(\mathcal{X})$ will denote the set of all bounded linear operators from $\mathcal{X}$ into $\mathcal{X}$. For $T\in L(\mathcal{X}) $, $T^{*}$ denotes the adjoint of $T$. $\pi:\mathcal{X} \longrightarrow \mathcal{X}^{**}$ is the natural canonical projection from $\mathcal{X}$ onto $\mathcal{X}^{**}$. A sequence space $S$ is called a \emph{BK-space} if it is a Banach space and the co-ordinate functionals are continuous on $S.$ That is the relations $x_n= \{ {\alpha_j}^{(n)}\}$, $x=\{ \alpha_j\} \in S$, $\lim\limits_{n \longrightarrow \infty}x_n=x$ imply $\lim\limits_{n \longrightarrow \infty}\alpha_j^{(n)}=\alpha_j ~~~(j=1,2,3,...).$

\section{Preliminaries}
  Gavruta \cite{12}, introduced the notion of $K$-frame and atomic system for an operator $K$ in a Hilbert space. She gave the following definition.

\begin{defn} \cite{12} Let $\mathcal{H}$ be a Hilbert space, $K\in L(\mathcal{H})$ and $\lbrace x_n\rbrace\subseteq \mathcal{H}$. Then \\
(a) $\lbrace x_n\rbrace$ is called a \emph{K-frame} for $\mathcal{H}$ if there exist constants $A,B>0$ such that
\begin{eqnarray*}
A\Vert K^*x\Vert^2\leq\sum\limits_{n=1}^{\infty}\vert\langle x,x_n\rangle\vert^2\leq B\Vert x\Vert^2, \ \text{for all} \ x\in \mathcal{H}.
\end{eqnarray*}
(b) $\lbrace x_n\rbrace$ is called an \emph{atomic system for $K$} if
\begin{enumerate}[(i)]
\item the series $\sum\limits_{n=1}^{\infty}c_nx_n$ converges for all $c=\lbrace c_n\rbrace\in l^2$.
\item there exists $C>0$ such that for every $x\in \mathcal{H}$ there exists $\lbrace a_n\rbrace\in l^2$ such that $\Vert \lbrace a_n\rbrace\Vert_{l^2}\leq C\Vert x\Vert$ and $K(x)=\sum\limits_{n=1}^{\infty}a_nx_n.$
\end{enumerate}
\end{defn}
 Shah \cite{SS} defined and studied approximative K-atomic decompositions and approximative $\mathcal{X}_d$-frames in Banach spaces. For further studies related to this concept one may refer \cite{SH, n}.

Next, we give some results which we will use throughout this manuscript.
\begin{defn}
\cite{16} Let $T\in L(\mathcal{X})$, we say that an operator S$\in L(\mathcal{X})$ is a pseudoinverse of $T$ if $TST=T$. Also, $S\in L(\mathcal{X})$ is the generalized inverse of $T$ if $TST=T$ and $STS=S$.
\end{defn}
\begin{lem}\cite{22}\label{PI} Let $\mathcal{X}$, $\mathcal{Y}$ be Banach spaces and $T:\mathcal{X} \longrightarrow \mathcal{Y}$ be a bounded linear operator. Then the following statements are equivalent:
\begin{enumerate}
\item There exist two continuous projection operators $P:\mathcal{X}\rightarrow \mathcal{X}$ and $Q:\mathcal{Y} \rightarrow \mathcal{Y}$ such that
\begin{eqnarray}\label{1E1}
P( \mathcal{X})=kerT \quad and \quad Q(\mathcal{Y})=T( \mathcal{X}).
\end{eqnarray}
\item There exist closed subspaces $W$ and $Z$ such that $\mathcal{X}=KerT\oplus W$ and $\mathcal{Y}=T(\mathcal{X})\oplus Z.$
\item $T$ has a pseudo inverse operator $T^\dag$.
\end{enumerate}
If two continuous projection operators $P:\mathcal{X} \rightarrow \mathcal{X}$ and $Q:\mathcal{Y} \rightarrow \mathcal{Y}$ satisfies (\ref{1E1}), then there exists a pseudo inverse operator $T^\dag$ of $T$ such that \begin{equation*}
T^\dag T=I_{\mathcal{X}}-P \ and \ TT^\dag=Q,
\end{equation*}
where $I_{\mathcal{X}}$ is the identity operator on $\mathcal{X}$.
\end{lem}
\begin{lem}\label{5}
\cite{2,19} Let $\mathcal{X}$ be a Banach space. If $T\in L(\mathcal{X})$ has a generalized inverse $S\in L(\mathcal{X}).$ Then $TS$,$ST$ are projections and $TS(\mathcal{X})=T(\mathcal{X})$ and $ST(\mathcal{X})=S(\mathcal{X}).$
\end{lem}

\begin{lem}\label{6}
\cite{15} Let $\mathcal{X}$ and $\mathcal{Y}$ be Banach spaces. Let $L:\mathcal{X} \rightarrow \mathcal{X}$ and $L_1:\mathcal{X} \rightarrow \mathcal{Y}$ be linear operators. Then the following conditions are equivalent:
\begin{enumerate}[(a)]
\item $L=L_2L_1$, for some continuous linear operator $L_2:L_1(\mathcal{X})\rightarrow \mathcal{X}$.
\item $\Vert L(\mathcal{X})\Vert\leq k\Vert L_1(\mathcal{X})\Vert$, for some $k\geq 0$ and for all $x\in \mathcal{X}$.
\item $Range(L^*)\subseteq Range(L_1^{*})$.
\end{enumerate}
\end{lem}
\begin{lem}\label{7}
\cite{1} Assume $T\in B(\mathcal{X},\mathcal{Y})$. If $S\in B(\mathcal{X},\mathcal{Z})$ with $Range(S^*)\subseteq Range(T^*)$ and $\overline{Range(T )}$ is complemented, then there
exists $V\in B(\mathcal{Y},\mathcal{Z})$ such that $S = VT.$
\end{lem}
\begin{lem}\label{8}
\cite{4} Let $\mathcal{X}_d$ be a BK-space for which the canonical unit vectors $\lbrace e_n\rbrace$ form a Schauder basis. Then the space $\mathcal{Y}_d=\lbrace\lbrace h(e_n)\rbrace\vert h\in \mathcal{X}_d^*\rbrace$ with norm $\Vert\lbrace h(e_n)\rbrace\Vert_{\mathcal{Y}_d}=\Vert h\Vert_{\mathcal{X}_d^*}$ is a BK-space isometrically isomorphic to $\mathcal{X}_d^*$. Also, every continuous linear functional $\Phi$ on $\mathcal{X}_d$ has the form
$\Phi\lbrace c_n\rbrace=\sum\limits_{n=1}^{\infty}c_nd_n,$
where $\lbrace d_n\rbrace\in \mathcal{Y}_d$ is uniquely determined by $d_n=\Phi(e_n)$, and
$\Vert \Phi\Vert=\Vert\lbrace \Phi(e_n)\rbrace\Vert_{\mathcal{Y}_d}$.
\end{lem}
Terekhin \cite{20} introduced and studied frames in Banach spaces.
\begin{defn}
\cite{20} Let $\mathcal{X}$ be a Banach space and $\mathcal{X}_d$ be a BK-space with the sequence of canonical unit vectors $\lbrace e_n\rbrace$ as basis. Let $Y_d$ be a sequence space mentioned in Lemma \ref{8}. A sequence $\lbrace x_n \rbrace_{n=1}^{\infty}\subseteq \mathcal{X}$ is called a frame with respect to $\mathcal{X}_d$ if
\begin{enumerate}[(a)]
\item $\lbrace f(x_n)\rbrace\in Y_d,$~ for all $f\in \mathcal{X}^*$,
\item there exist constants $A$ and $B$ with $0<A\leq B<\infty$ such that
\begin{eqnarray}\label{1}
A\Vert f\Vert_{\mathcal{X}^*}\leq\Vert\lbrace f(x_n)\rbrace\Vert_{\mathcal{Y}_d}\leq B\Vert f\Vert_{\mathcal{X}^*}, \ \text{for  all} \ f\in \mathcal{X}^*.
\end{eqnarray}
\end{enumerate}
We refer (\ref{1}) as the \emph{frame inequality}. If at least (a) and the upper bound condition in (\ref{1}) are satisfied, then $\lbrace x_n\rbrace$ is called a \emph{Bessel sequence} for $\mathcal{X}$ with respect to $\mathcal{X}_d$.
A frame $\lbrace x_n\rbrace$ is called exact if on removal of its one element $x_n$, it is no longer a frame for $\mathcal{X}$.
\end{defn}

\section{Approximative atomic system for operators}

 Gavruta \cite{12}, introduced and studied atomic system for an operator $K$ in Hilbert spaces. Poumai and Jahan \cite{KS} defined and studied atomic systems for operators in Banach spaces as a generalization of atomic system in Hilbert spaces. Here we have generalize this concept further and introduce the concept of approximative atomic systems for operators in  Banach spaces and obtain new and interesting results. We starts this section with the following definition of approximative atomic system for K:
\begin{defn}
Let $\mathcal{X}$ be a Banach Space and $K$ be a bounded linear operator in $\mathcal{X}$. A sequence $\lbrace x_n\rbrace\subseteq \mathcal{X}$ is called an \emph{approximative atomic system for K} in $\mathcal{X}$ with respect to $\mathcal{X}_d$ if
\begin{enumerate}[(a)]
\item $\lbrace x_n\rbrace$ is a Bessel sequence for $\mathcal{X}$ with respect to $\mathcal{X}_d,$
\item there exists an approximative $\mathcal{X}_d$-Bessel Sequence $\{h_{n,i}\}\underset{n \in \mathbb{N}}{_{i=1,2,3,...,m_n}}\subset \mathcal{X}^*$ for $\mathcal{X}$ such that
\begin{eqnarray*}
K(x)=\lim\limits_{n \rightarrow \infty}\sum\limits_{i=1}^{m_n}h_{n,i}(x)x_i, \ \text{for all} \ x\in \mathcal{X}.
\end{eqnarray*}
\end{enumerate}
\indent The sequence $\{h_{n, i}\}\underset{n \in \mathbb{N}}{_{i=1,2,3,...,m_n}}$ is called the associated approximative $\mathcal{X}_d$-Bessel sequence.
\end{defn}
\textbf{Observation}: If $\{x_n\}$ is an atomic system for $K$ in $\mathcal{X}$ with respect to $\mathcal{X}_d,$ then for $h_{n,i}=f_i,~~i=1,2,3,...,n,~~n \in \mathbb{N},$ $\{x_n\}$ is an approximative atomic system for $K.$
\begin{defn}
Let $\mathcal{X}$ be a Banach Space and $M$ be a closed subspace of $\mathcal{X}$. A sequence $\lbrace x_n \rbrace\subseteq \mathcal{X}$ is called \emph{an approximative family of local atoms} for $M$ if
\begin{enumerate}[(a)]
\item $\lbrace x_n \rbrace_{n=1}^{\infty}$ be a Bessel sequence for $\mathcal{X}$ with respect to $\mathcal{X}_d,$
\item there exists an approximative $\mathcal{X}_d$-Bessel Sequence $\{h_{n,i}\}\underset{n \in \mathbb{N}}{_{i=1,2,3,...,m_n}}\subset \mathcal{X}^*$ for $M$ such that
\begin{eqnarray*}
x=\lim\limits_{n \rightarrow \infty}\sum\limits_{i=1}^{m_n}h_{n,i}(x)x_i, \ \text{for all} \ x\in M.
\end{eqnarray*}
\end{enumerate}
The sequence $\{h_{n,i}\}\underset{n \in \mathbb{N}}{_{i=1,2,3,...,m_n}}$ is called the associated approximative $\mathcal{X}_d$-Bessel sequence.
\end{defn}
\textbf{Observations}~(I) If $\{x_n\}$ is a family of local atoms for $M$, then for $h_{n,i}=f_i,~~i=1,2,3,...,n,~~n \in \mathbb{N},$ $\{x_n\}$ is an approximative family of local atoms for $M.$\\
\\~~~(II)
Let $\mathcal{X}$ be a Banach space and $M$ be a closed subspace of $\mathcal{X}$. Let $\lbrace x_n \rbrace$ be an approximative family of local atoms for $M$ with respect to $\mathcal{X}_d$ and let $\{h_{n,i}\}\underset{n \in \mathbb{N}}{_{i=1,2,3,...,m_n}}$ be its associated approximative $\mathcal{X}_d$-Bessel sequence with bound $B$. Then, $(x_n, \{h_{n,i}\}\underset{n \in \mathbb{N}}{_{i=1,2,3,...,m_n}})$ is an approximative atomic decomposition for $M$ with respect to $\mathcal{X}_d$ and $\lbrace x_n\rbrace$ is a frame for $M$.
 Indeed, by given hypotheses,
\begin{eqnarray*}
x=\lim\limits_{n \rightarrow \infty}\sum\limits_{i=1}^{m_n}h_{n,i}(x)x_i, \text{for all} \ x \in M.
\end{eqnarray*}
Let $A$ be the bound of Bessel sequence $\lbrace x_n\rbrace.$ Then for $x \in M$
\begin{eqnarray*}
\Vert x\Vert&=&\sup\limits_{h\in  \mathcal{X}^*,\Vert h\Vert=1}\vert \lim\limits_{n \rightarrow \infty}\sum\limits_{i=1}^{m_n}h_{n,i}(x)h(x_n)\vert
\leq A\Vert\lbrace h_{n,i}(x)\rbrace\Vert
\end{eqnarray*}
With Similar argument, we have
\begin{eqnarray*}
1/B\Vert f\Vert\leq\Vert\lbrace f(x_n)\rbrace\Vert, \ \text{for all} \ f\in M^*.
\end{eqnarray*}
\\~(III)
Let $\mathcal{X}$ be a Banach space and $K\in L(\mathcal{X})$. If $\lbrace x_n\rbrace$ is an approximative atomic system for $K$ in $\mathcal{X}$. Then, there exist constants $C,D>0$ such that
\begin{eqnarray*}
C\Vert K(x)\Vert\leq\Vert\lbrace h_{n,i}(x)\rbrace\Vert, \text{for all} \ x\in \mathcal{X},
\end{eqnarray*} and
\begin{eqnarray*}
D\Vert K^*(f)\Vert\leq\Vert\lbrace h(x_n)\rbrace\Vert, \text{for all} \ h\in \mathcal{X}^*.
\end{eqnarray*}
\\~(IV)
$\lbrace x_n\rbrace\subseteq \mathcal{X}$ is a Bessel sequence for $\mathcal{X}$ with respect to $\mathcal{X}_d$ if and only if there exists a bounded linear operator $T$ from $\mathcal{X}_d$ into $\mathcal{X}$ for which $T\lbrace h_n\rbrace=\sum\limits_{n=1}^{\infty}h_nx_n$, $\lbrace h_n\rbrace\in \mathcal{X}_d.$ Recall that $T$ is called the \emph{synthesis operator} associated with the Bessel sequence $\{x_n\},$ and the bounded linear operator $R:\mathcal{X}^*\rightarrow \mathcal{Y}_d$ given by
\begin{eqnarray*}
R(h)=\lbrace h(x_n)\rbrace, \ \text{for} \ h\in \mathcal{X}^*,
\end{eqnarray*}
is called the \emph{analysis operator} of Bessel sequence $\lbrace x_n\rbrace$. Also observe that from Lemma \ref{8}, $\mathcal{Y}_d$ is isometrically isomorphic to $\mathcal{X}_d^*$. Let $J_d:\mathcal{X}_d^*\rightarrow \mathcal{Y}_d$ denote an isometric isomorphism from $\mathcal{X}_d^*$ onto $\mathcal{Y}_d$.
Note that the synthesis operator $T$ need not be onto. Indeed, let $\mathcal{X}_d=\mathcal{X}=l_1$ and let $\lbrace e_n\rbrace$ be a sequence of canonical unit vectors as basis of $\mathcal{X}$. Take $x_n=e_{n+1}$, for $n \in \mathbb{N}$. Let $h=\lbrace h_n\rbrace_{n=1}^\infty\in  \mathcal{X}^*=l^\infty$ then $\lbrace h(x_n)\rbrace\in l^\infty$ and $\Vert\lbrace h(x_n)\rbrace\Vert_{l^\infty}\leq A\Vert h\Vert_{l^\infty}$, where $A>0$ is some constant. Thus, $\lbrace x_n\rbrace$ is Bessel sequence for $\mathcal{X}$. But $T:\mathcal{X}_d\rightarrow \mathcal{X}$ given by $T(e_n)=x_n$, for $n \in \mathbb{N}$ is a bounded linear operator which is not onto.
\\
$\lbrace x_n\rbrace\subseteq \mathcal{X}$ is a frame if and only if there exists a bounded linear operator $T$ from $\mathcal{X}_d$ onto $\mathcal{X}$ for which $T\lbrace c_n\rbrace=\sum\limits_{n=1}^{\infty}c_nx_n$, $\lbrace c_n\rbrace\in \mathcal{X}_d$.

From the frame inequality, $T^*$ is one-one and Range of $T^*$ is closed in $\mathcal{X}_d^*$. So by [\cite{17}, page 103], $T$ is onto.\\
\indent Conversely, if $T$ is onto. Then by [\cite{17}, page 103] $T^*$ is one-one and range of $T^*$ is closed. Also, by [\cite{9}, page 487], there exists a constant $D>0$ such that
\begin{eqnarray*}
\Vert h\Vert\leq D\Vert T^*(h)\Vert=D\Vert \lbrace h(x_n)\rbrace\Vert,\ \text{for all} \ h\in \mathcal{X}^*.
\end{eqnarray*}

In the following result, we construct an approximative family of local atoms for $K(\mathcal{X})$ and an approximative atomic decomposition for $[K(\mathcal{X})]^*$ from a given approximative atomic system for a bounded linear operator $K$.
\begin{thm}\label{T1}
Let $\lbrace x_n\rbrace$ be an approximative atomic system for $K$ in $\mathcal{X}$  and $K$ has pseudo inverse $K^\dagger$. Then, $\lbrace x_n\rbrace$ is an approximative family of local atoms for $K(\mathcal{X})$ in $\mathcal{X}$.
Moreover, if $\mathcal{X}_d^*$ has a sequence of canonical unit vectors $\lbrace e_n^*\rbrace$ as basis, then there exists an approximative $\mathcal{X}_d$-Bessel sequence $\{h_{n,i}\}\underset{n \in \mathbb{N}}{_{i=1,2,3,...,m_n}}\subset \mathcal{X}^*$ for $K(\mathcal{X})$ such that $(h_{n,i},\pi(x_n))$ is an approximative atomic decomposition for $[K(\mathcal{X})]^*$ with respect to $\mathcal{X}_d^*$.
\end{thm}
\begin{proof}
Since $\lbrace x_n\rbrace$ is an approximative atomic system for $K$ in $\mathcal{X}$  $K$ has a pseudo inverse $K^\dagger$, $KK^\dagger$ is a projection from $\mathcal{X}$ onto $K(\mathcal{X})$ and
$KK^\dagger(x)=x, \ \text{for all} \ x\in K(\mathcal{X}).$
Thus, $KK^\dagger\vert_{K(\mathcal{X})}=I_{K(\mathcal{X})}$. Let $\lbrace h_{n,i}\rbrace$ be its associated approximative $\mathcal{X}_d$-Bessel sequence with bound C. Take $f_{n,i}=(K^\dagger\vert_{K(\mathcal{X})})^*(h_{n,i})$, $n\in\NN$ and
let $x\in K(\mathcal{X})$. Then, we compute
\begin{eqnarray*}
x&=&K(K^\dagger\vert_{K(\mathcal{X})}(x))
\\
&=&\lim\limits_{n \rightarrow \infty}\sum\limits_{i=1}^{m_n}h_{n,i}(K^\dagger\vert_{K(\mathcal{X})}(x))x_i
\\
&=&\lim\limits_{n \rightarrow \infty}\sum\limits_{i=1}^{m_n}(K^\dagger\vert_{K(\mathcal{X})^*}h_{n,i}(x))x_i
\\
&=&\lim\limits_{n \rightarrow \infty}\sum\limits_{i=1}^{m_n}f_{n,i}(x)x_i
\end{eqnarray*}
\indent Now, we will show that $\lbrace f_{n,i}\rbrace$ is an approximative Bessel sequence for $K(\mathcal{X})$ with respect to $\mathcal{X}_d$. Let $x\in K(\mathcal{X}).$ Then we have
\begin{eqnarray*}
\lbrace f_{n,i}(x)\rbrace=\lbrace h_{n,i}(K^\dagger\vert_{K(\mathcal{X})}(x)\rbrace\in \mathcal{X}_d, \ \text{for all} \ x\in K(\mathcal{X})
\end{eqnarray*}
and
\begin{eqnarray*}
\Vert\lbrace f_{n,i}(x)\rbrace\Vert&=&\Vert\lbrace h_{n,i}(K^\dagger\vert_{K(\mathcal{X})}(x))\rbrace\Vert\leq C\Vert K^\dagger\vert_{K(\mathcal{X})}(x)\Vert
\\
&\leq&C\Vert K^\dagger\Vert\Vert x\Vert, \ \text{for all} \ x\in K(\mathcal{X}).
\end{eqnarray*}
Also, $(K^\dagger\vert_{K(\mathcal{X})})^* K^*=I^*_{[K(\mathcal{X})]^*}$. Let $h \in [K(\mathcal{X})]^*$.  Then
\begin{eqnarray*}
\Vert h\Vert&=&\Vert(K^\dagger\vert_{K(\mathcal{X})})^* K^*(h)\Vert\leq \Vert K^\dagger\Vert\Vert K^*(h)\Vert
\\
&=&\Vert K^\dagger\Vert\sup\limits_{x\in K(\mathcal{X}),\Vert x\Vert=1}\vert h(\lim\limits_{n \rightarrow \infty}\sum\limits_{i=1}^{m_n} h_{n,i}(x)x_i)\vert
\\
&\leq& C\Vert K^\dagger\Vert\Vert \lbrace h(x_i)\rbrace\Vert.
\end{eqnarray*}
Let $h\in [K(\mathcal{X})]^*.$ Then for $N\in \mathbb{N},$ we have
\begin{eqnarray*}
\Vert h-\lim\limits_{n \rightarrow \infty}\sum\limits_{i=1}^{m_n}h(x_i)h_{n,i}\Vert&=&\sup\limits_{x\in K(\mathcal{X}),\Vert x\Vert=1}\vert h(x)- \lim\limits_{n \rightarrow \infty}\sum\limits_{i=1}^{m_n}h(x_i)h_{n,i}(x)\vert
\\
&=&\sup\limits_{x\in K(\mathcal{X}),\Vert x\Vert=1}\vert\lim\limits_{n \rightarrow \infty}\sum\limits_{i=1}^{m_n}h(x_i)h_{n,i}(x)\vert
\\
&\leq&C\Vert K^\dagger\Vert\Vert \lim\limits_{n \rightarrow \infty}\sum\limits_{i=1}^{m_n}h(x_i)e_i^*\Vert
\\
&\rightarrow& 0 \ \text{as N} \rightarrow \infty.
\end{eqnarray*}
\indent Thus, $f=\lim\limits_{n \rightarrow \infty}\sum\limits_{i=1}^{m_n}h(x_i)h_{n,i}$, for all $h\in [K(\mathcal{X})]^*$.
\end{proof}
If $\lbrace x_n\rbrace$ is an approximative atomic system for $K$ and $K$ has pseudo inverse $K^\dagger$, then for $h_{n,i}=f_i$, $\lbrace x_n\rbrace$ is frame as well as an approximative family of local atoms for $K(\mathcal{X})$.\\
The following theorem gives a necessary condition for the existence of an approximative atomic system for a bounded linear operator $K$.
\begin{thm}\label{T2}
If $\lbrace x_n\rbrace$ is an approximative atomic system for $K$, then there exists a bounded linear operator $T: \mathcal{X}_d\rightarrow \mathcal{X}$ such that $T(e_n)=x_n$, $n\in\NN$ and Range$K^{**}\subseteq RangeT^{**}$, where $\lbrace e_n\rbrace$ is the sequence of canonical unit vectors as a basis of $\mathcal{X}_d$.
\end{thm}
\begin{proof}
Since $\lbrace x_n\rbrace$ is an approximative atomic system for $K$, $T:\mathcal{X}_d\rightarrow \mathcal{X}$ given $T(\lbrace h_{n,i}\rbrace) =\lim\limits_{n \rightarrow \infty}\sum\limits_{i=1}^{m_n}h_{n,i}x_i$ is a well defined bounded linear operator such that $T(e_i)=x_i$, $n\in\NN$. Also, there exists an approximative $\mathcal{X}_d$-Bessel sequence $\lbrace h_{n,i}\rbrace$ for $\mathcal{X}$ with bound $B$ such that
\begin{eqnarray*}
K(x)=\lim\limits_{n \rightarrow \infty}\sum\limits_{i=1}^{m_n}h_{n,i}(x)x_i, \ \text{for all} \ x\in \mathcal{X}.
\end{eqnarray*}
Then, for $h\in \mathcal{X}^{*}$, we have
\begin{eqnarray*}
\Vert K^*(h)\Vert&=&\sup\limits_{x\in \mathcal{X},\Vert x\Vert=1}\vert K^*h(x)\vert=\sup\limits_{x\in \mathcal{X},\Vert x\Vert=1}\vert\lim\limits_{n \rightarrow \infty}\sum\limits_{i=1}^{m_n}h_{n,i}(x)h(x_i)\vert
\\
&\leq&B\Vert \lbrace h(x_i)\rbrace\Vert=B\Vert T^*(h)\Vert.
\end{eqnarray*}
Hence, by Lemma \ref{7}, we have $RangeK^{**}\subseteq RangeT^{**}$.
\end{proof}
Gavruta \cite{12}, had proved that for a separable Hilbert space $\mathcal{H}$ a sequence $\lbrace x_n\rbrace\subseteq \mathcal{H}$ is an atomic system for $K$ if and only if there exists a bounded linear operator $L:l^2\rightarrow \mathcal{H}$ such that $L(e_n)=x_n$ and $Range(K)\subseteq Range(L)$, where $\lbrace e_n\rbrace$ is an orthonormal basis for $l^2$. Towards the converse of the Theorem \ref{T2}, we have the following theorem.
\begin{thm}
Let $\mathcal{X}_d$ be reflexive, $\mathcal{X}_d^*$ has a sequence of canonical unit vectors $\lbrace e_n^*\rbrace$ as basis and $\lbrace x_n\rbrace$ be a sequence in $\mathcal{X}$. Let $K\in L(\mathcal{X})$, $T:\mathcal{X}_d\rightarrow  \mathcal{X}$ be bounded linear operator with $T(e_n)=x_n,~n \in \mathbb{N}$, $RangeK^{**}\subseteq RangeT^{**}$ and $\overline{T^*(\mathcal{X}^*)}$ is complemented subspace of $\mathcal{X}_d^*$. Then $\lbrace x_n\rbrace$ is an approximative atomic system for K.
\end{thm}
\begin{proof}
Since $T:\mathcal{X}_d\rightarrow \mathcal{X}$ is a bounded linear operator and is given by
\begin{eqnarray*}
T(\lbrace \alpha_n\rbrace)=\sum\limits_{n=1}^{\infty}\alpha_nx_n,~~\lbrace \alpha_n\rbrace \in \mathcal{X}_d^*.
\end{eqnarray*}
So, by Observation IV, $\lbrace x_n\rbrace$ is a Bessel sequence with bound say $B$. Since $RangeK^{**}\subseteq RangeT^{**}$ and $\overline{T^*(\mathcal{X}^*)}$ is complemented subspace of $\mathcal{X}_d^*$, by Lemma \ref{7}, there exists a bounded linear operator $\theta:\mathcal{X}_d^*\rightarrow \mathcal{X}^*$ such that $K^*=\theta T^*$.
 Take $h_{n,i},~~\underset{n \in \mathbb{N}}{_{i=1,2,3,...,m_n}}=\theta(e_n^*)$, $n\in\NN$. Then, for $x\in \mathcal{X}$, we get
\begin{eqnarray*}
\theta^*(x)(e_j^*)=\pi(x)\theta(e_j^*)=h_{n,i}(x);i=1,2,...{m_n} \ \text{for} \ n\in\NN.
\end{eqnarray*}
This gives $\lbrace h_{n,i}(x)\rbrace=\theta^*(x)\in \mathcal{X}_d, \ \text{for all} \ x\in \mathcal{X}$
and $\Vert\lbrace h_{n,i}(x)\rbrace\Vert=\Vert\theta^*(x)\Vert\leq\Vert \theta\Vert\Vert x\Vert, \ \text{for all} \ x\in \mathcal{X}.$
Also, for $h\in \mathcal{X}^*,$ we have
\begin{eqnarray*}
K^*(h)=\theta(\lbrace h(x_n)\rbrace)=\theta(\sum\limits_{n=1}^{\infty}h(x_n)e_n^*)
&=&\lim\limits_{n \rightarrow \infty}\sum\limits_{i=1}^{m_n}h(x_i)h_{n,i}.
\end{eqnarray*} and for $x\in \mathcal{X}$ and $n\in \NN,$ we compute
\begin{eqnarray*}
\Vert K(x)-\lim\limits_{n \rightarrow \infty}\sum\limits_{i=1}^{n}h_{n,i}(x)x_i\Vert&=&\sup\limits_{h\in \mathcal{X}^*, \Vert h\Vert=1}\vert h(K(x))-\lim\limits_{n \rightarrow \infty}\sum\limits_{i=1}^{n}h_{n,i}(x)h(x_i)\vert
\\
&=&\sup\limits_{h\in \mathcal{X}^*, \Vert h\Vert=1}\vert \sum\limits_{i=n+1}^{\infty}h_{n,i}(x)h(x_i)\vert
\\
&\leq&B\Vert \sum\limits_{i=n+1}^{\infty}h_{n,i}(x)e_i\Vert\rightarrow 0 \ \text{as} \ n\rightarrow\infty.
\end{eqnarray*}
\indent Hence, $K(x)=\lim\limits_{n \rightarrow \infty}\sum\limits_{i=1}^{m_n}h_{n,i}(x)x_i$, for all $x\in \mathcal{X}$.
\end{proof}
The following theorem gives a complete characterization for the approximative family of local atoms for the closed subspace of a Banach space.
\begin{thm}
Let $\lbrace x_n\rbrace$ be a Bessel sequence for $\mathcal{X}$ with respect to $\mathcal{X}_d$, $\mathcal{X}_d^*$ has a sequence of canonical unit vectors $\lbrace e_n^*\rbrace$ as basis and $M$ be a closed subspace of $\mathcal{X}$. Let there exists a projection P from $\mathcal{X}$ onto M. Then, the following statements are equivalent:
\begin{enumerate}[(a)]
\item $\lbrace x_n\rbrace$ is an approximative family of local atoms for M.
\item $\lbrace x_n\rbrace$ is an approximative atomic system for P.
\item There exists a bounded linear operator $U:M \rightarrow \mathcal{X}_d$ such that $TUP=P$, where $T$ is the synthesis operator of the Bessel sequence $\lbrace x_n\rbrace$.
\end{enumerate}
\end{thm}
\begin{proof}
$(a)\Rightarrow(b)$ Since $\lbrace x_n\rbrace$ is an approximative family of local atoms for M so, there exists an approximative $\mathcal{X}_d$-Bessel sequence $\{h_{n,i}\}\underset{n \in \mathbb{N}}{_{i=1,2,3,...,m_n}}$ for $M$ such that
\begin{eqnarray*}
x=\lim\limits_{n \rightarrow \infty}\sum\limits_{i=1}^{m_n}h_{n,i}(x)x_i, \ \text{for all} \ x\in M.
\end{eqnarray*}
Now for each $y\in \mathcal{X}$, $P(y)\in M$ and
\begin{eqnarray*}
P(y)=\lim\limits_{n \rightarrow \infty}\sum\limits_{i=1}^{m_n}h_{n,i}(P(y))x_i=\lim\limits_{n \rightarrow \infty}\sum\limits_{i=1}^{m_n}(P^*(h_{n,i})(y)x_i, \ \text{for all} \ y\in \mathcal{X}.
\end{eqnarray*}
\indent Clearly $\lbrace P^*(h_{n,i})\rbrace$ is an approximative $\mathcal{X}_d$-Bessel sequence for $\mathcal{X}$. Hence, $\lbrace x_n\rbrace$ is an approximative atomic system for $P$.
\\
$(b)\Rightarrow (a)$ Since $\{x_n\}$ is an approximative atomic system for P and $x=P(x)$, for all $x\in M$, the proof follows.
\\
$(a)\Rightarrow (c)$ Since $\{x_n\}$ is an approximative family of local atoms for M so there exists an approximative $\mathcal{X}_d$-Bessel sequence $\{h_{n,i}\}\underset{n \in \mathbb{N}}{_{i=1,2,3,...,m_n}}$ for M such that
\begin{eqnarray*}
P(x)=\lim\limits_{n \rightarrow \infty}\sum\limits_{i=1}^{m_n}h_{n,i}(P(x))x_i, \ \text{for all} \ x\in \mathcal{X}.
\end{eqnarray*}
Let $U:M\rightarrow \mathcal{X}_d$ be the analysis operator for $\lbrace h_{n,i}\rbrace$ given by $U(P(x))=\lbrace h_{n,i}(P(x))\rbrace,x\in \mathcal{X}.$ Also, $\lbrace h_{n,i}(P(x)\rbrace\in \mathcal{X}_d$, for all $x\in \mathcal{X}$. Hence
\begin{eqnarray*}
P(x)&=&\lim\limits_{n \rightarrow \infty}\sum\limits_{i=1}^{m_n}h_{n,i}(P(x))x_i=T(\lbrace h_{n,i}(P(x))\rbrace)
\\
&=&TUP(x), \ \text{for all} \ x\in \mathcal{X}.
\end{eqnarray*}\\
$(c)\Rightarrow (a)$ Since $U:M\rightarrow \mathcal{X}_d$ is a bounded linear operator such that $TUP=P$. Take $h_{n,i}=U^*(e_n^*),n\in\NN$ and $x\in M$. Then $h_{n,i}(x)=U^*(e_n^*)(x)=e_n^*(U(x)).$
Therefore, $\lbrace h_{n,i}(x)\rbrace=U(x)\in \mathcal{X}_d$, for all $x\in M$ and
\begin{eqnarray*}
\Vert \lbrace h_{n,i}(x)\rbrace\Vert=\Vert U(x)\Vert\leq\Vert U\Vert\Vert x\Vert, \ \text{for all} \ x\in M.
\end{eqnarray*}
Also, for each, $x \in M,$ there exists $y\in \mathcal{X}$ such that $x=P(y)$. Moreover, $\lbrace h_{n,i}(P(y))\rbrace\in \mathcal{X}_d$ for all $y\in \mathcal{X}$ and
\begin{eqnarray*}
x=TUP(y)=T(\lbrace h_{n,i}(P(y))\rbrace)=\lim\limits_{n \rightarrow \infty}\sum\limits_{i=1}^{m_n}h_{n,i}(P(y))x_i=\lim\limits_{n \rightarrow \infty}\sum\limits_{i=1}^{m_n}h_{n,i}(y)x_i
\end{eqnarray*}
\end{proof}
Next, we shall show that if $\lbrace x_n\rbrace$ is an approximative $K$ atomic system with $K=I_\mathcal{X}$, then every complemented subspace of $\mathcal{X}$ has an approximative family of local atoms and in case $h_{n,i}=f_n$ then $\{x_n\}$ is a $K$ atomic system and every complemented subspace of $\mathcal{X}$ has a family of local atoms.
\begin{thm}
Let $\lbrace x_n\rbrace$ be a frame for $\mathcal{X}$ with respect to $\mathcal{X}_d$. If there exists an approximative $\mathcal{X}_d$-Bessel sequence $\lbrace h_{n,i}\rbrace$ such that
\begin{eqnarray*}
x=\lim\limits_{n \rightarrow \infty}\sum\limits_{i=1}^{m_n}h_{n,i}(P(x))x_i, \ \text{for all} \ x\in \mathcal{X},
\end{eqnarray*}
then every complemented subspace of $\mathcal{X}$ has an approximative family of local atoms.
\end{thm}
\begin{proof}
Let $\{x_n\}$ be a frame for $\mathcal{X}$ and $T$ be the synthesis operator of $\{x_n\}$. Then by given hypothesis, there exists an approximative $\mathcal{X}_d$-Bessel sequence $\lbrace h_{n,i}\rbrace$ such that
\begin{eqnarray*}
x=\lim\limits_{n \rightarrow \infty}\sum\limits_{i=1}^{m_n}h_{n,i}(x)x_i, \ \text{for all} \ x\in \mathcal{X}.
\end{eqnarray*}
Let $S$ be the analysis operator of $\lbrace h_{n,i}\rbrace$. Note that $I_\mathcal{X}=TS$. Let $M$ be a complemented subspace of $\mathcal{X}$ and $P$ be the projection from $\mathcal{X}$ onto $M$. Take $y_n=P(x_n)$, for $n\in \mathbb{N}$. Define $T_1:\mathcal{X}_d\rightarrow M$ by $T_1=P\circ T$.
Then clearly $T_1$ is a bounded linear operator from $\mathcal{X}_d$ onto $M$ such that
\begin{eqnarray*}
T_1(\lbrace h_{n,i}\rbrace)=P(\lim\limits_{n \rightarrow \infty}\sum\limits_{i=1}^{m_n}h_{n,i}(x)x_i)=\lim\limits_{n \rightarrow \infty}\sum\limits_{i=1}^{m_n}h_{n,i}(x)y_i,~~ \lbrace h_{n,i}\rbrace\in \mathcal{X}_d.
\end{eqnarray*}
Thus, $\lbrace y_n\rbrace$ is frame for $M$ and a Bessel sequence for $\mathcal{X}$ with respect to $\mathcal{X}_d$. Let $N=S(M)$ and let $\alpha=\lbrace h_{n,i}\rbrace\in \mathcal{X}_d$. Then, $T_1(\alpha)=x,$ for some $x\in M.$ Now take $\alpha_1=S(x).$ Then, $\alpha_1\in N.$ Also, if $\alpha_0=\alpha-\alpha_1,$ then
\begin{eqnarray*}
T_1(\alpha_0)=T_1(\alpha)-T_1(\alpha_1)
=x-P(x)=0.
\end{eqnarray*}
This gives $\alpha_0 \in kerT_1.$ Thus $\alpha=\alpha_0+\alpha_1$. Then $\lambda\in N\cap kerT_1$. Then $T_1(\lambda)=0$ and there exists $x_1\in M$ such that $\lambda=S(x_1)$. Therefore
\begin{eqnarray*}
x_1=P(x_1)=PTS(x_1)=T_1(\lambda)=0.
\end{eqnarray*}
Thus, $\lambda=0$ and so $N\cap kerT_1=\lbrace 0\rbrace$. Therefore $\mathcal{X}_d=N\oplus kerT_1$. Now, by Lemma \ref{PI}, $T_1$ has a pseudo inverse $T_1^\dagger$. Moreover $T_1T_1^\dagger$ is a projection from $M$ onto $T_1(\mathcal{X}_d)=M$. This gives $T_1T_1^\dagger=I_M$. Let $\lbrace z_n\rbrace$ be the sequence of coordinate functionals on $\mathcal{X}_d$. Take $f_{n,i}=(T_1^\dagger)^*(l_n)$, for $n\in \mathbb{N}$ and let $x\in M$. Then
\begin{eqnarray*}
f_{n,i}(x)=(T_1^\dagger)^*(z_n)(x)=z_n(T_1^\dagger(x)), \ \text{for} \ n\in\mathbb{N}, \ x\in M.
\end{eqnarray*}
Therefore $\lbrace f_{n,i}(x)\rbrace=T_1^\dagger(x)\in \mathcal{X}_d$, for all $x\in M$ and
\begin{eqnarray*}
\Vert\lbrace f_{n,i}(x)\rbrace\Vert=\Vert T_1^\dagger(x)\Vert\leq\Vert T_1^\dagger\Vert\Vert x\Vert, \ \text{for all} \ x\in M.
\end{eqnarray*}
Thus $\lbrace f_{n,i}\rbrace$ is an approximative $\mathcal{X}_d$-Bessel sequence for $M$. Therefore, we compute
\begin{eqnarray*}
x=T_1T_1^\dagger(x)=T_1(\lbrace f_{n,i}(x)\rbrace)=\lim\limits_{n \rightarrow \infty}\sum\limits_{i=1}^{m_n}f_{n,i}(x)y_i, \ \text{for all} \ x\in M.
\end{eqnarray*}
\end{proof}
Let $\lbrace x_n\rbrace$ be a Bessel sequence for $\mathcal{X}$ with respect to $\mathcal{X}_d$ and let T be its associated synthesis operator. Also, let $\lbrace h_{n,i}\rbrace$ be an approximative $\mathcal{X}_d$-Bessel sequence for $\mathcal{X}$ with its associated analysis operator S. Then $TS: \mathcal{X} \rightarrow \mathcal{X}$ is an operator such that
\begin{eqnarray*}
TS(x)=\lim\limits_{n \rightarrow \infty}\sum\limits_{i=1}^{m_n}h_{n,i}(x)x_i, \ \text{for all} \ x\in \mathcal{X}.
\end{eqnarray*}
Note that $\lbrace x_n\rbrace$ is an approximative atomic system for $K$ in $\mathcal{X}$ with its associated approximative $\mathcal{X}_d$-Bessel sequence $\lbrace h_{n,i}\rbrace$ if and only if
\begin{eqnarray}\label{E2}
TS=K.
\end{eqnarray} Also one may observe that in case $\{h_{n,i}\}=\{f_i\}$ then $\{x_n\}$ is an atomic system for $K$ in $\mathcal{X}$ with its associated $\mathcal{X}_d$-Bessel sequence $\{f_i\}$\\
The following theorem gives a method how to construct an approximative atomic system for $K$ from a given Bessel sequence.
\begin{thm}
Let $\mathcal{X}_d^*$ has a sequence of canonical unit vectors $\lbrace e_n^*\rbrace$ as basis. Let $\lbrace x_n\rbrace$ be a Bessel sequence for $\mathcal{X}$ with respect to $\mathcal{X}_d$ and $T:\mathcal{X}_d\rightarrow \mathcal{X}$ be its synthesis operator such that T has pseudoinverse $T^\dagger$. Let $K\in  L(\mathcal{X})$ with $K(\mathcal{X})\subseteq T(\mathcal{X}_d)$. Then $\mathcal{X}$ has an approximative atomic system for $K.$
\end{thm}
\begin{proof}
Since $T$ has a pseudo inverse $T^\dagger$, by Lemma \ref{5},  $TT^\dagger$ is a projection from $\mathcal{X}$ onto $T(\mathcal{X}_d)$. Define $S:\mathcal{X}\rightarrow \mathcal{X}_d$ as
\begin{eqnarray}\label{E3}
S=T^\dagger K+W-T^\dagger TW,
\end{eqnarray}
where $W:\mathcal{X} \rightarrow \mathcal{X}_d$ is a bounded linear operator. Then, we compute
\begin{eqnarray*}
TS&=&T(T^\dagger K+W-T^\dagger TW)
\\
&=&TT^\dagger K+TW-TT^\dagger TW
\\
&=&TT^\dagger K=K.
\end{eqnarray*}
Note that $T^*=J_d^{-1}\circ R$ and $\lbrace J_de_n^*\rbrace$ is a basis of $\mathcal{Y}_d$.
 Taking $h_{n,i}=S^*(e_n^*)$, $f_{n,i}=(T^+)^*(e_n^*)$ and $g_n=W^*(e_n^*)$ , for $n\in \mathcal{N},$ we compute
\begin{eqnarray*}
T^*(T^\dagger)^*e_n^*=T^*(f_{n,i})=J_d(R(f_{n,i}))=\lim\limits_{n \rightarrow \infty}\sum\limits_{i=1}^{m_n}f_{n,i}(x_i)e_i^*
\end{eqnarray*}
and
\begin{eqnarray*}
h_{n,i}&=&K^*(T^\dagger)^*(e_n^*)+W^*(e_n^*)-W^*T^*(T^\dagger)^*(e_n^*)
\\
&=&K^*f_{n,i}+g_i-\lim\limits_{n \rightarrow \infty}\sum\limits_{i=1}^{m_n}f_{n,i}(x_i)g_i, \ \text{for all} \ n\in \mathbb{N}.
\end{eqnarray*}
Also, for $x\in \mathcal{X}$, we have
\begin{eqnarray*}
h_{n,i}(x)=S^*(e_n^*)(x)=e_n^*(S(x)), \ \text{for all} \ n\in \mathbb{N}.
\end{eqnarray*}
Therefore $\lbrace h_{n,i}(x)\rbrace=S(x)\in \mathcal{X}_d$, for all $x \in \mathcal{X}$. Also,
\begin{eqnarray*}
\Vert\lbrace h_{n,i}(x)\rbrace\Vert&=&\Vert S(x)\Vert
\\
&\leq&\Vert T^\dagger K+W-T^\dagger TW\Vert\Vert x\Vert,~x \in \mathcal{X}.
\end{eqnarray*}
Thus, $\lbrace h_{n,i}\rbrace$ is an approximative $\mathcal{X}_d$- Bessel sequence and
\begin{eqnarray*}
\lim\limits_{n \rightarrow \infty}\sum\limits_{i=1}^{m_n}h_{n,i}(x)x_i&=&T(\lbrace h_{n,i}(x)\rbrace)=T(\lbrace S^*(e_n^*)(x)\rbrace)
\\
&=&T(\lbrace e_n^*(S(x))\rbrace)=TS(x)
\\
&=&K(x), \ x\in \mathcal{X}.
\end{eqnarray*}
Hence, $\lbrace x_n\rbrace$ is an approximative atomic system for $K$.
\end{proof}
Note that (\ref{E2}) is the general formula for all $U$ satisfying the equality (\ref{E3}). As a result of this, let $U_\circ$ be a linear operator such that $TU_\circ=K$ and take $W=U_\circ$ in the right side of (\ref{E3}), then
\begin{eqnarray*}
T^\dagger K+U_\circ-T^\dagger TU_\circ=T^\dagger K+U_\circ-T^\dagger K=U_\circ
\end{eqnarray*}
\indent The following theorem shows that we can construct an approximative atomic system for $K$ from a given approximative $\mathcal{X}_d$-Bessel sequence.
\begin{thm}
 Let $\mathcal{X}_d$ be a BK-space and let $\lbrace e_n\rbrace$ be a sequence of canonical unit vectors as basis of $\mathcal{X}_d$ and $\mathcal{X}_d^*$ has a sequence of canonical unit vectors $\lbrace e_n^*\rbrace$ as basis. Let $\lbrace h_{n,i}\rbrace$ be an approximative $\mathcal{X}_d$-Bessel sequence for $\mathcal{X}$ with $S:\mathcal{X} \rightarrow \mathcal{X}_d$ as its analysis operator and S has pseudo inverse $S^\dagger$. Let $K\in L(\mathcal{X})$ and $K^*(\mathcal{X}^*)=[h_{n,i}]$. Then, $\mathcal{X}$ has an approximative atomic system for $K.$
\end{thm}
\begin{proof} Since $\lbrace h_{n,i}\rbrace$ is an approximative $\mathcal{X}_d$-Bessel sequence for $\mathcal{X}$ with $S:\mathcal{X} \rightarrow \mathcal{X}_d$ as its analysis operator and S has pseudoinverse $S^\dagger.$
Define $T:\mathcal{X}_d \rightarrow \mathcal{X}$ by
\begin{eqnarray}\label{E4}
T=KU^\dagger +W(I-UU^\dagger),
\end{eqnarray}
where $W:\mathcal{X}_d\rightarrow \mathcal{X}$ is a bounded linear operator. Then, we compute
\begin{eqnarray*}
TS&=&(KS^\dagger +W(I-SS^\dagger))S
\\
&=&KS^\dagger S+WS-WSS^\dagger S
\\
&=&KS^\dagger S.
\end{eqnarray*}
Also, for $x\in \mathcal{X}$, $S^*:\mathcal{X}_d^*\rightarrow \mathcal{X}^*$ is given by
\begin{eqnarray*}
S^*(e_n^*)(x)=e_n^*(S(x))=h_{n,i}(x), \ n\in \mathcal{N}.
\end{eqnarray*}
So, we have
\begin{eqnarray*}
[h_{n,i}]=[S^*(e_n^*)]=S^*(\mathcal{X}_d^*)=K^*(\mathcal{X}^*).
\end{eqnarray*}
Also $S$ has a pseudo inverse $S^\dagger$, $S^*$ has pseudo inverse $(S^\dagger)^*$. So,
\begin{eqnarray*}
S^*(S^\dagger)^*S^*(\mathcal{X}_d^*)=S^*(\mathcal{X}_d^*)=K^*(\mathcal{X}^*).
\end{eqnarray*}
Thus, we conclude that $S^*(S^\dagger)^*K^*=K^*$. Therefore, $TS=KS^\dagger S=K$.
Let $x_n=T(e_n)$, $y_n=S^\dagger(e_n)$ and $l_n=W(e_n)$, for $n\in \mathbb{N}$.  Then
\begin{eqnarray*}
x_n=KS^\dagger(e_n)+W(I-SS^\dagger)(e_n)
\\
=K(y_n)+l_n-\lim\limits_{n \rightarrow \infty}\sum\limits_{i=1}^{m_n}h_{n,i}(y_i)l_i, ~~n\in \mathcal{N}.
\end{eqnarray*}
Also, $T: \mathcal{X}_d\rightarrow \mathcal{X}$ is a bounded linear operator given by
\begin{eqnarray*}
T(\lbrace h_{n,i}\rbrace)=\lim\limits_{n \rightarrow \infty}\sum\limits_{i=1}^{m_n}h_{n,i} x_i,~~ \lbrace h_{n,i}\rbrace\in \mathcal{X}_d.
\end{eqnarray*}
By observation III, $\lbrace x_i\rbrace$ is a Bessel sequence. Thus, we get
\begin{eqnarray*}
K(x)=TS(x)=T(\lim\limits_{n \rightarrow \infty}\sum\limits_{i=1}^{m_n}h_{n,i}(x)e_i)
=\lim\limits_{n \rightarrow \infty}\sum\limits_{i=1}^{m_n}h_{n,i}(x)x_i, \ x \in \mathcal{X}.
\end{eqnarray*}
\indent Hence $\lbrace x_n\rbrace$ is an approximative atomic system for K.
\end{proof}

\end{document}